\documentclass[11pt]{article}
\usepackage{graphicx,amsmath,latexsym,amssymb,cancel, comment}
\usepackage{tikz}
\usepackage{tikz-cd}
\usetikzlibrary{arrows.meta}
\usepackage{authblk}

\usepackage{capt-of}

\usepackage{mathdots}
\usepackage[psdextra,plainpages=false,pdfpagelabels,colorlinks=false,urlcolor=black,pdfpagemode=UseNone,pdfstartview=FitH]{hyperref}
\usepackage{amsthm}
\usepackage[overload]{empheq}
\usepackage{cleveref}
\usepackage{fancyhdr}
\usepackage{float}
\usepackage{enumerate}   
\newtheorem{theorem}{Theorem}[section]

\newtheorem{definition}[theorem]{Definition}
\newtheorem{corollary}[theorem]{Corollary}
\newtheorem{lemma}[theorem]{Lemma}
\newtheorem{example}[theorem]{Example}

\newtheorem*{theorem*}{Theorem}

\newcommand{\Z}{{\mathbb Z}}

\newcommand{\sm}{\backslash}

 \newcommand{\set}[1]{\left\{#1\right\}}

\numberwithin{equation}{section}

\newcommand{\N}{{\mathbb{N}}}

\newcommand{\B}{BM(F)}

\newcommand{\jjoin}{\stackrel{j}\asymp}
\newcommand{\jsim}{\stackrel{j}\sim}
\newcommand{\ora}[1]{\overrightarrow{#1}}
\newcommand{\ds}[1]{\displaystyle {#1}}

\renewcommand{\sm}{\setminus}

\newcommand{\njoin}[1]{\stackrel{#1}\asymp}
\newcommand{\nfoldjoin}[2]{\underset{#1}{\stackrel{#2}{\asymp}}}

\title{
Algebraic Structures on Graphs Joined by Edges}
\author{
Daniel Pinzon, 
Daniel Pragel,
\&
Joshua  Roberts  }
\affil{Department of Mathematics and Statistics, Georgia Gwinnett College}
\date{}

\fancypagestyle{footer}{%
  \fancyhf{}
  
  \fancyfoot[L]{{ \footnotesize\noindent\textbf{Acknowledgements}: We are grateful for the work of undergraduate students Shahriyar Roshan Zamir \& Hope Doherty who assisted with example calculations.}}
}

\begin{document}
\maketitle
\begin{abstract}
Let $G_1\jjoin G_2$, the $j$-join of two graphs, be the union of two disjoint graphs connected by $j$ edges in a one-to-one manner. In previous work by Gyurov and Pinzon \cite{BP}, which generalized the results of Badura \cite {Bad} and Rara \cite{R}, the determinant of the adjacency matrix of two $j$-joined graphs was decomposed to sums of determinants of these graphs with vertex deletions or directed graph handles. In this paper, we find the necessary and sufficient properties of a graph $G$ so that for any graph $H$, the determinant of $G\jjoin H$ and $H \jjoin G$ is equal to the determinant of $H$. Subsequently, we define a homomorphism from a quotient of graphs with the $j$-join operation to the monoid of integer matrices under multiplication. We demonstrate through examples that this homomorphism allows us to more easily calculate determinants of chains of joined graphs. This generalizes the work done on determinants of grids and cylinders done in \cite{DanP}, \cite{Bien}, \cite{Deift} and \cite{BKT}
\end{abstract}

\thispagestyle{footer}

\noindent

\section{Introduction}

Let $\mathbb{G}_{j}$ be the set of labeled finite simple directed graphs with at least $2j$ vertices such that, for $G \in \mathbb{G}_{j}$, the set of vertices is  $V(G)=\{1, 2, \dots, m\}$, where $m$ is the number of vertices of $G$. The set of edges is $E(G) \subseteq \{(v,w) \mid v,w \in V, v\neq w\}$. We will at times refer to the ``last" vertices of $G$ using the convention $-1,-2,-3,\dots$ for $|V(G)|,|V(G)|-1,|V(G)|-2,\dots$. Throughout this paper we will use the term graph in general to be a directed graph.

Let $G, H \in \mathbb{G}_{j}$ where $V(G)=\set{1,\dots,m}$ and $V(H)=\set{1,\dots,n}$.   Following \cite{BP}, we define the \textit{j-join} of  $G$ and $H$ with $j$ edges as the graph formed by joining each of the ``last" distinct $j$ vertices of $G$ with each of the corresponding ``first" $j$ distinct vertices of $H$ in both directions. See \Cref{join1} below.

\begin{figure} [h!] \label{join1}
    \centering
    \includegraphics[width = 1
    \textwidth]{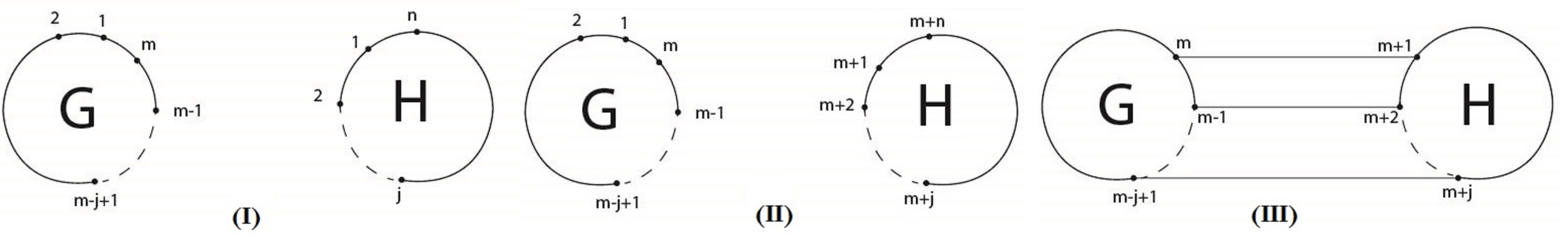}
    \caption{$j$-operation}
    \label{EXDT}
\end{figure}

There is a choice as to whether to put the instructions of the joining operation on the graph elements or on the operation. The labeling contains the information on where to join one graph with another so that we have one well-defined operation rather than many operations. This motivates the use of labeled graphs. It will be clear from the definition below that this operation is associative.

\begin{definition}\label{jopr} 
 Let $G, H \in \mathbb{G}_{j}$. The $j$-join of $G$ and $H$, denoted $ G\jjoin H$, has the vertex and edge sets below where $|V(G)|=m$ and each $i\in V(H)$ is relabeled as $m+i$ in $V(G\jjoin H)$.

$$V(G \stackrel{j}\asymp H) = V(G) \cup \{m+i\mid i\in V(H)\} \text{, and}$$
\begin{align*}
    E(G \stackrel{j}\asymp H) = E(G) &\cup \{(m+1 -i, m+i),(m+i, m+1 -i) \mid 1 \le i \le j\}\\
    &\cup  \set{(m+i,m+j)\mid (i,j)\in E(H)}. 
\end{align*} 
\end{definition} 

A motivational problem that this paper solves in \Cref{idsection} is to determine the necessary and sufficient conditions for the existence of an ``identity'' graph in the sense that left or right joining such a graph to any graph $G$ does not change the determinant of $G$, that is,
$|G \jjoin Id \jjoin H| = |G \jjoin H|$. We note that we have adopted the notation $|G|$ to mean the determinant of the adjacency matrix of $G$. It can the case that a graph can act as a one-sided identity only as seen in the following example. 
\begin{figure}[H]
\begin{center}
\begin{tikzpicture}[scale=0.5,dot/.style 2 args={circle,inner sep=1pt,fill,label={#2:#1},name=#1}]
\tikzstyle{every node}=[circle,fill=black,inner sep=0.5mm,yshift=0,xshift=0]
            \node  [label={[xshift=0cm ]left:1}] (1) at  (-4,0)  {};
            \node  [label={[xshift=0cm ]above:2}] (2) at  (-2,0)  {};
            \node  [label={[xshift=0cm ]right:3}] (3) at  (0,1)  {};
            \node  [label={[xshift=0cm ]right:4}] (4) at  (0,-1)  {};
            \draw (1)--(2);
            \draw (2)--(3);
            \draw (3)--(4);
            \draw (2)--(4);
\end{tikzpicture}
\end{center}
\caption{Example of right identity for 1-join.}
\label{rid}
\end{figure}
\begin{example}
    Consider the graph in \Cref{rid}, which we will denote $R_{id}$, joined by one edge to $K_3$. We see that $|R_{id}|=1$ and $|K_3|=2$. Then, $| K_3 \njoin{1} R_{id}|=2=|K_3|$ but $|R_{id} \njoin{1} K_3|=4$ 
    Results in this paper will show that $R_{id}$ will act as an right identity, but not as a left identity.
\end{example}

From \cite{BP}, we can write the determinant of the adjacency matrix of the $j$-join of two graphs $G$ and $H$ as a sum of determinants of the adjacency matrices of modifications of $G$ and $H$.  The two modifications are vertex deletions and directed graph handles.  We describe them below.
  
\begin{definition}\label{verdel}
{\bf (Vertex deletion)} For a graph $G$ and a vertex $v\in V(G)$ we denote by $G\backslash\{v\}$ the subgraph of $G$
obtained by removing the vertex $v$ from $V(G)$ and all edges that are incident with $v$ from $E(G).$

Further, if $R$ is a subset of vertices of $G$, we denote by $G\sm R$ the subgraph of $G$
obtained by deleting all vertices in $R$ from $G$.
\end{definition}

The operation of attaching a directed graph handle is attaching a copy of a directed path on 3 vertices, $\ora{P}_3$, as described below.

\begin{definition}
{\bf{(Directed Graph Handle)}}\label{handle}
For a graph $G$ and vertices $u,v\in V(G)$, we denote $G_{[u,v]}$ to be the graph where a new vertex
$w=|V(G)|+1$ is added to $V(G)$ and a directed edge from vertex $u$ to vertex $w$ and a directed edge from vertex $w$ to vertex $v$ are added to $E(G)$.
Vertex $w$ is called the \textit{directed graph handle vertex} of the directed graph handle $[u,v]$.

Further, if $B$ is a set of ordered pairs of elements of $V(G)$, we denote $G_B$ to be the graph
where for each $[u,v]\in B$ a new directed graph handle is attached to $G$.
\end{definition} 

We denote $|R|$ as the number of elements of $R$ and $|B|$ as the number of handles in $B$. In this paper, we will regularly attach handles on graphs where a set of vertices $R$ have been removed. If $B$ is the set of handles, then we denote this graph as $(G\sm R)_B$. 

We will be using the main result from \cite{BP} given below which describes how to express the determinant of $G\jjoin H$ as a sum of the determinants of modifications of graphs. We sum over all possible vertex removals $R$ from a set $J=\{i\mid 1\leq i\leq j\}$ which are the ``first" $j$ vertices of $H$. We also modify $G$ in a similar, conjugated way. That is, if $i\in R$ is removed from $H$, then the conjugated vertex $-i\in R^*=\{-i\mid i\in R\}$ is removed from $G$. 

Given $B$, the set of appended directed graph handles on $H$, the conjugate set is defined as
$B^*=\{[-r,-c]\ \mid \ -r,-c \in V(G),[c,r]\in B\}.$
Note that the direction of the conjugated handles are reversed. Below we give an example of a term in the sum.

\begin{example} Let $G, H \in \mathbb{G}_{10}$ and $m=|V(G)|$. Then, $J=\{1,\dots,10\}.$ Let $R=\{4\}$ and let $B=\{[1,3],[2,5]\}$ be handles made from $J\sm R$. Then, the conjugate removal set and handle set for $G$ are $R^*=\{-4\}=\{m-3\}$ and $B^*=\{[-3,-1],[-5,-2]\}=\{[m-2,m],[m-4,m-1]\}$. The figure below shows $(G\sm R^*)_{B^*}$ and $(H\sm R)_{B}$. Note that the open circle at $m-3\in V(G)$ and $4\in V(H)$ correspond to vertex deletions at those vertices. Also, notice that the handles on $G$ are in the opposite direction of those on $H$.
\begin{center}
\begin{tikzpicture}[scale=2,dot/.style 2 args={circle,inner sep=1pt,fill,label={#2:#1},name=#1}]
    \draw   (-4,0) circle (.75cm);
    \draw   (0,0) circle (.75cm);
\tikzstyle{every node}=[circle,fill=white,inner sep=0.5mm]
    \node  (G) at  (-4,0) {{\fontsize{50}{50}\selectfont\mbox{{$G$}}}};
\tikzstyle{every node}=[circle,fill=white,inner sep=0.5mm]
    \node  (H) at  (0,0) {{\fontsize{50}{50}\selectfont\mbox{{$H$}}}};
\tikzstyle{every node}=[circle,fill=black,inner sep=2pt,xshift=-8cm]
\node  at  (35:0.75)   {}   ;
\tikzstyle{every node}=[circle,fill=white,inner sep=1.5pt,xshift=-8cm]
\node  [label={[xshift=8cm ]left:$m-3$}](3) at  (35:0.75)   {}   ;
\tikzstyle{every node}=[circle,fill=black,inner sep=2pt,xshift=-8cm]
\node  [label={[xshift=8cm ]left:$m-4$}](3) at  (60:0.75)   {}   ;
\node  [label={[xshift=8cm ]left:$m-2$}](2) at  (15:0.75) {};
\node (v1) at (15:1.5) {};
\node (v2) at (-15:1.5) {};
\node [label={[xshift=8cm ]left:$m-1$}](1) at (-15:0.75){};
\node  [label={[xshift=8cm ]left:$m$}](0) at  (-45:0.75) {};
\draw[-{Latex[length=3mm]}] (3)--(v1);
\draw[-{Latex[length=3mm]}] (v1)--(1);
\draw[-{Latex[length=3mm]}] (2)--(v2);
\draw[-{Latex[length=3mm]}] (v2)--(0);
\tikzstyle{every node}=[circle,fill=black,inner sep=2pt,xshift=0cm]
\node [label={[xshift=0cm ]right:4}] (h3) at  (150:0.75)  {}   ;
\tikzstyle{every node}=[circle,fill=white,inner sep=1.5pt,xshift=0cm]
\node at  (150:0.75)  {}   ;
\tikzstyle{every node}=[circle,fill=black,inner sep=2pt,xshift=0cm]
\node [label={[xshift=0cm ]right:5}] (h4) at  (120:0.75)  {}   ;
\node  [label={[xshift=0cm ]right:3}](h2) at  (165:0.75) {};
\node (hv1) at (165:1.5) {};
\node (hv2) at (195:1.5) {};
\node [label={[xshift=0cm ]right:2}](h1) at (195:0.75){};
\node  [label={[xshift=0cm ]right:1}](h0) at  (225:0.75) {};
\draw[-{Latex[length=3mm]}] (h1)--(hv1);
\draw[-{Latex[length=3mm]}] (hv1)--(h4);
\draw[-{Latex[length=3mm]}] (h0)--(hv2);
\draw[-{Latex[length=3mm]}] (hv2)--(h2);
\end{tikzpicture}
\end{center}
    
\end{example}
 
The determinant of the disjoint union of the two graphs above represents (up to sign) one of the terms in the summation below. The summation does not sum over all possible handles, but rather a subset of them. An \textit{allowable handle set} is a set of handles $B$ such that a vertex appears only once in the set and, for any two handles $[a,b], [c,d]$, either $[a,b]<[c,d]$ or $[a,b]>[c,d]$, where the inequality is component-wise.

The theorem below is a full generalization of ideas first introduced by Rara \cite{R} $(j=1)$ and continued in \cite{Bad} and \cite{BG} $(j=2)$.

\begin{theorem}\label{sum}\cite{BP}
Let $G$ and $H$ be graphs of order $m$ and $n$ respectively where. Let $J$ be the set of vertices of $H$ that are joined to $G$. Then
$$
|G \jjoin H|  = \sum_{R \subset J}\sum_B (-1)^{|R|+|B|}|(G \sm R^*)_{B^*}|  |(H \sm R)_{B}|
$$
where the summation is over all allowable handle sets $B$.

\end{theorem}

The following corollary, which is equivalent to Theorem 1 in \cite{R}, is immediate by the preceding theorem. 

\begin{corollary}\label{cor}
$$|G \stackrel{1} \asymp H| = |G||H| - |G \setminus \{-1\}| \cdot |H \setminus \{1\}|.$$
\end{corollary}

The next corollary and proceeding figure below illustrate how to express the 2-join of two graphs using the \Cref{sum}.
\begin{corollary}

\begin{equation*}
    \begin{split}
    |G \stackrel{2} \asymp H| =& |G||H| - |G \setminus \{-1\}| \cdot |H \setminus \{1\}| - |G \setminus \{-2\}| \cdot |H \setminus \{2\}| \\
    &+ |G \setminus \{-1,-2\}| \cdot |H \setminus \{1,2\}| \\
    &- |G_{[-1,-2]}| \cdot |H_{[2,1]} - |G_{[-2,-1]}| \cdot |H_{[1,2]}| 
    \end{split}
\end{equation*}
\end{corollary}

\begin{figure}[H]
    \centering
    \includegraphics[width = 1 \textwidth]{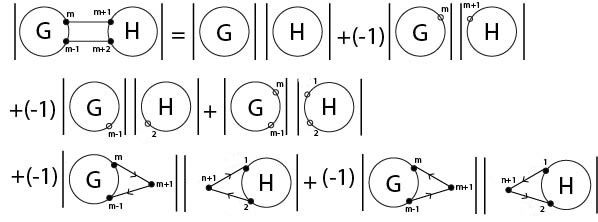}
    \caption{Determinant of $G \overset{2}{\asymp} H$}
    \label{twoconnect}
\end{figure}

Intuitively, the sum can be understood following Harary’s interpretation in \cite{H} of the determinant as the sum of the determinants of all spanning disjoint directed cycles. Either the joining edge is not used in a directed cycle (as in the first term in \Cref{twoconnect}), the edge is its own directed 2-cycle (as in the second through fourth terms in \Cref{twoconnect}), or the edge is part of a pair that has parts of its cycle in both graphs (as in the last two terms). 

\section{An Equivalence Relation}


For a given positive integer $j$, we can define an equivalence relation $\jsim$. We use the terms of the summation in \Cref{sum} as a motivation for this definition.

\begin{definition}\label{jrel}
For $G, H \in \mathbb{G}_j$ and for any vertex deletion sets $R_1,R_2$ and corresponding allowable handle sets $B_1, B_2$, $G \jsim H$  if and only if
$$ 
      |(G \setminus R_1\cup R_2^*)_{B_1\cup B_2^*}| = |(H \setminus R_1\cup R_2^*)_{B_1\cup B_2^*}|.\label{eq1}
$$
\end{definition}

Clearly this is an equivalence relation on $\mathbb{G}_j$. We define the quotient $\mathcal{G}_j = \mathbb{G}_j/\jsim$. The following theorem shows that the $j$-join is a well-defined operation on the quotient set. As a consequence of the following theorem, we can define an induced $j$-join product on the equivalence classes, 

$$[G]\jjoin [H]=[G\jjoin H].$$

\begin{theorem}\label{well-defined}
For a fixed integer $j>0$, the $j$-join is a well-defined binary operation over the equivalence relation $\jsim$.
\end{theorem}
\begin{proof}
Let $G_1, G_2, H_1, H_2 \in \mathbb{G}_j$ where we recall that their orders are greater than $2j$. Assume that 
$G_1 \jsim G_2 $ and $H_1 \jsim H_2$. We need to show that 
$(G_1\jjoin H_1) \jsim (G_2 \jjoin H_2)$.

Using \Cref{jrel}, we consider
$|((G_1\jjoin H_1) \setminus R_1\cup R_2^*)_{B_1\cup B_2^*}|$. We note that since since $G_1$ has at least $2j$ vertices and $R_1, B_1$ act on the first $j$ vertices of $G_i \jjoin H_i$, which are  not the  vertices involved in the $j$-join. Similarly, the conjugate sets only act on $H_i$. Using this fact and \Cref{sum}, we have
\begin{eqnarray*}
&&|((G_1\jjoin H_1) \setminus R_1\cup R_2^*)_{B_1\cup B_2^*}|\\ &=&
|(G_1\setminus R_1)_{B_1})\jjoin (H_1 \setminus R_2^*)_{B_2^*}|\\
&=& \displaystyle{\sum_{R \subset J}}\sum_B (-1)^{|R|+|B|}|(G_1 \setminus R_1\cup R^* )_{ B_1\cup B^*}| |(H_1 \setminus   R \cup R_2^*)_{ B\cup B_2^*})|\\
&=&\displaystyle{\sum_{R \subset J}}\sum_B (-1)^{|R|+|B|}|(G_2 \setminus  R_1\cup R^*)_{B_1\cup B^* }| |(H_2 \setminus   R\cup R_2^*)_{ B\cup B_2^*})|\\
&=&|((G_2\jjoin H_2) \setminus R_1\cup R_2^*)_{B_1\cup B_2^*}|,
\end{eqnarray*}
where the third equality follows from the fact that $G_1 \jjoin G_2$ and $H_1 \jjoin H_2$.
\end{proof}

\subsection{Equivalence Classes}
In this section, we will explore the necessary conditions for the existence of an identity element, a zero, and other equivalence classes. For simplicity, we shall use $[G]$ instead of $[G]_j$ when the $j$ is clear from the context.
\subsubsection{The Identity Class}{\label{idsection}}
  We say that $I\in \mathbb{G}_j$ is an \textit{identity graph} if and only if for any $G\in\mathbb{G}_j$, $[I\jjoin G] =[G\jjoin I] = [G]$. That is, that joining any graph by an identity graph does not change its equivalence class and thus, in particular, preserves its determinant. 

\begin{theorem}\label{ID}
Let $j\in\N$ and $I\in \mathbb{G}_j$. If for any vertex deletion sets $R_1,R_2$ and corresponding handle sets $B_1, B_2$, 
$$|(I\sm{R_1\cup R_2^*})_{B_1\cup B_2^*}|=
            \begin{cases} 
              (-1)^{|R_1|+|B_1|} &\text{if } R_1=R_2, B_1=B_2 \\
              0&\text{ otherwise},
            \end{cases}$$
then $I$ is an identity graph.
\end{theorem}

\begin{proof}
Let $j \in \N, G \in \mathbb{G}_j,$ and $I$ be as above. Then using a similar argument as in the proof of \Cref{well-defined},
\begin{equation*}
\begin{split}
&|(I\jjoin G) \sm (R_1\cup R_2^*))_{B_1\cup B_2^*}|
=|(I\sm R_1)_{B_1}\jjoin (G\sm R_2^*)_{B_2^*}|\\
&=\sum_{B}\sum_{R \subset J} (-1)^{|R|+|B|}|(I \sm (R_1\cup R^*)_{(B_1\cup B^* )}| |(G \sm ( R\cup R_2^*))_{B\cup B_2^*}|\\
&=(-1)^{|R_1|+|B_1|}|(I \sm (R_1\cup R_1^*)_{(B_1\cup B_1^*)}| |(G \sm (R_1\cup R_2^*))_{B_1\cup B_2^*}|\\
&= |(G \sm (R_1\cup R_2^*))_{B_1\cup B_2^*}|.
\end{split}
\end{equation*}

The third equality results from the properties of $I$ causing all terms of the summation to vanish except where $R=R_1$and $B = B_1$. Thus we have shown that if $I$ satisfies the conditions of the theorem, then $[I\jjoin G] = [G]$. Similarly, $[G] = [G \jjoin I]$. 
\end{proof}

For the 1-join, any path graph of order $4k$ is a member of the identity class. 

The following figure demonstrates the existence of representatives of the identity classes for the $j-$join operation for $j\geq1$. The graph $G$ has $m$ vertices labeled $j+1$ through $m+j$.
\begin{example}{$j$-Join Identity}\label{jjoinid}
\begin{center}
\begin{tikzpicture}[scale=1]
\draw (0,0) circle (1.5cm);
\node (C) at (0:0) {$|G|=(-1)^j$};
\tikzstyle{every node}=[circle,fill=black,inner sep=0.5mm,yshift=0,xshift=0]
\node  [label={[xshift=0cm ]left:1}] (1) at  (4,1)  {};
\node [label={[xshift=0cm ]right:m+2j}] (m) at (6,1) {};
\node  [label={[xshift=0cm ]left:j}] (2) at  (4,-1)  {};
\node [label={[xshift=0cm ]right:m+j+1}] (m-1) at (6,-1) {};
\tikzstyle{every node}=[circle,fill=black,inner sep=0.25mm,yshift=0,xshift=0]
\node () at (5,.8) {};
\node () at (5,.6) {};
\node () at (5,.4) {};
\node () at (5,-.8) {};
\node () at (5,-.6) {};
\node () at (5,-.4) {};
\tikzstyle{every node}=[circle,fill=white,inner sep=0.25mm,yshift=0,xshift=0]
\node [label={[xshift=0cm,yshift=-1cm ]above:$j$ copies of $P_2$}] () at (5,-.25) {};

\draw (1)--(m);
\draw (2)--(m-1);
\end{tikzpicture}
\end{center}
\captionof{figure}{Directed edges from the copies of $P_2$ to $G$ are not necessary unless connectivity is desired.}\label{jiden}
\end{example}

\begin{theorem}\label{id-exist}
    The graph in \Cref{jiden} is a representative of the identity class for any $j$.
\end{theorem}

\begin{proof}
Consider any graph $G$, with $|G|=(-1)^j$ and order $m$. Let $I_j$ be a disjoint union of $G$ with $j$ copies of $P_2$. We label the vertices of $I_j$ as shown in \Cref{jiden}. Since $|P_2|=-1$, then it is clear that $|I_j|=1$.

Next, note that in $(I_j\setminus R_1 \cup R_2^*)_{B_1 \cup B_2^*}$, if $R_1 \neq R_2$, then either there exists at least one copy of $P_2$ where one vertex is removed and the other is isolated or the single vertex is part of a handle. In the former case, we have that the determinant is zero. In the latter case, we use Harary's definition from \cite{H} of the determinant of a graph  and note that there are no spanning directed cycle decompositions of $I_j$. Thus in order for $(I_j\setminus R_1 \cup R_2^*)_{B_1 \cup B_2^*}$ to be nonzero it must be the case that $R_1 = R_2$.

Now, assume $R_1 = R_2$. Then, each vertex in $R_1$ results in the deletion of a copy of $P_2$ which results in a change in the determinant by a factor of $-1$. Thus there will be a total change of $(-1)^{|R_1|}$ to the determinant.

Suppose that we have $(I_j\setminus R_1 \cup R_1^*)_{B_1 \cup B_2^*}$ with $B_1 \neq B_2$ and $\det((I_j\setminus R_1 \cup R_1^*)_{B_1 \cup B_2^*}) \neq 0.$ Then, there exists at least one handle $[a,b]$ in $B_1$ or $B_2$ that is not in the other handle set. Since $|(I_j\setminus R_1 \cup R_1^*)_{B_1 \cup B_2^*}| \neq 0$, then the vertex of that handle must be part of a directed cycle $C$ that is a subgraph of $(I_j\setminus R_1 \cup R_1^*)_{B_1 \cup B_2^*}$. This implies that $C$ contains at least 4 copies of $P_2$ from $I_j$. Let $n$ be the least positive integer such that $n \in V(C)$. We observe that vertex $n$ must be contained in one of the copies of $P_2$. So, either $(n, -n)$ or $(-n, n) \in E(C)$. Suppose $(n, -n) \in E(C)$. 

We observe that, beginning with vertex $n$, to travel along the directed cycle $C$, we move through a copy of $P_2$, then a handle, then followed by another copy of $P_2$ (in the opposite direction), and so on until arriving back at $n$. Formally, let $a_1, a_2, \dots, a_k \in \{1, 2, \dots, j\}$ be the positive vertices in $P=I_j\cap C$ of the directed cycle starting and ending at $n,$ in the order that they appear in $C$, so that $a_1 = a_k = n$. It follows  that, for odd $i \in \{1, \dots, k-1\}$, $(a_i, -a_i) \in E(C)$ and the handle on $P$ from $-a_i$ to $-a_{i+1}$ is traversed in $C$. Additionally, for even $i \in \{1, \dots, k-1\}$, $(-a_i, a_i) \in E(C)$ and the handle on $P$ from $a_i$ to $a_{i+1}$ is traversed in $C$. Note that $k\geq 4$.

 Furthermore, since $n$ is the minimum of $\{a_1, a_2, \dots, a_{k-1}\}$, $a_1 < a_3$ and $a_{k-2} > a_k$. This implies that there must be some $i \in \{ 1, \dots, k-3\}$ with $a_i < a_{i+2}$ and $a_{i+1} > a_{i+3}$. But then the handles $[a_i, a_{i+1}]$ and $[a_{i+2}, a_{i+3}]$ must both be in $B_1$ or $B_2$. But this implies that $B_1$ or $B_2$ is not an allowable handle set.

The case when $(-n,n) \in E(C)$ is handled analogously. Thus, we see that, if $B_1$ and $B_2$ are allowable handle sets and $\det((I_j\setminus R_1 \cup R_2^*)_{B_1 \cup B_1^*}) \neq 0,$ it must be the case that $B_1 = B_2.$
\end{proof}

Thus, we see that $(\mathcal{G}_j,\jjoin)$ is a monoid with identity as given above.

\section{Algebraic Structure of the Join Operation}

\begin{definition} 
  Let $S$ be a semigroup and $x,y\in S$. A \textbf{sandwich operation} $\bullet$ on $S$ is defined as $x\bullet y=xay$, where $a$ is an element of $S$.
\end{definition}

Hickey, in \cite{JBH}, showed that $S$ under this sandwich operation is also a semigroup denoted $(S,a)$ which we will call the sandwich semigroup on $a$.

Let $G$ be a representative of any equivalence class of graphs $[G]\in \mathcal{G}_j=\mathbb{G}_j/\jsim$ for some $j\in\N$. From the above, the equivalence class under the $j$-join is defined by all the modifications of $G$ given by $(G\sm R_1\cup R_2^*)_{B_1\cup B_2^*}$ where $R_i\subseteq  J = \set{1,...,j}$ and $B_i$ are handle sets on the vertices $J\sm R_i$ where the handles satisfy the conditions discussed in \Cref{sum}. This also defines the conjugate sets $R_2^*, B_2^*$.  

Now, consider the set of all allowable removal set and handle set pairs $$\set{(R,B)\mid R\subseteq J, B\subseteq \mathcal{P}({J-R\choose2}), B \text{ allowable}}.$$ From \cite{BP}, looking at the Laplacian expansion,  we see that there are $k=$ $2j\choose j$ elements. Let us enumerate these pairs as $\set{(R_1,B_1)=(\emptyset,\emptyset), (R_2,B_2)), ..., (R_k,B_k)}$. The choice of how these are numbered is arbitrary, but we need to fix a convention here for the rest of the paper. Now, define the functions $r_i:\mathcal{G}_j \rightarrow \Z$ as $r_i([G])=\det((G\sm R_i)_{B_i})$, where for $1\leq i\leq k$. We see that this function is well defined since any representative of the class will have the same determinant. Equivalently, we can define the conjugate modifications as $c_l:\mathcal{G}_j \rightarrow \Z$ as $c_l([G])=\det((G\sm R_l^*)_{B_l^*})$, $1\leq l\leq k$.\\

Now, we can define an operation from any graph equivalence class to a $k \times k$ integer matrix $\phi:\mathcal{G}_j\rightarrow M(k,\Z)$ as 
$$\phi([G])=[m_{il}]=[\det((r_i\circ c_l) (G))].$$


We can now prove the main results of our paper that the algebraic structure of the $j$-join under the $j$-join equivalence class is a sub-semigroup of the sandwich semigroup isomorphic to the $k \times k$ integer matrices under matrix multiplication. We begin with the following theorem that shows that our function is a homomorphism.
\begin{theorem}\label{main-thm}
Let $[G],[H]\in  \mathcal{G}_j$ and $\phi$ as in the discussion above. Then, 
$$\phi([G]\jjoin [H])=\phi([G\jjoin H])=\phi([G]) E_j \phi([H])$$
where $E_j=\left[e_{il}=
\begin{cases}
    (-1)^{|R|+|B|}& i=l\\
    0& \textrm{otherwise}
\end{cases}
\right].
$
\end{theorem}
\begin{proof}
Let $[m_{il}] = \phi([G\jjoin H]) = [\det((r_i\circ c_l) (G\jjoin H))]$ as defined above. Then, since $|V(G)|, |V(H)|\geq 2j$, then $r_i(G\jjoin H)=r_i(G)\jjoin H$ and $c_l(G\jjoin H) = G\jjoin c_l(H)$ and so $(r_i\circ c_l) (G\jjoin H) = r_i(G)\jjoin c_l(H) = (G\sm R_i)_{B_i}\jjoin (H\sm R_l^*)_{B_l^*}$ for some vertex removal sets $R_i, R_l$ and handle sets $B_i, B_l$.  

Then, calculating the determinant using \Cref{sum}, we have 
\begin{align*}
m_{il}=\det((r_i\circ c_l) (G\jjoin H))  = \det((G\sm R_i)_{B_i}\jjoin (H\sm R_l^*)_{B_l^*})=\\
\displaystyle{\sum_B\sum_{R \subset J}} (-1)^{|R|+|B|}|(G \sm R_i\cup R^*)_{B_i\cup B^*}| |(H \sm R\cup R_l^*)_{B\cup B_l^*})| 
\end{align*}
Note that the vertices of $R_i^*, B_i^*$ are never the same as those of $R,B$ since the number of vertices of $G$ is at least $2j$. We see this is the same for $H$ as well.

Notice that the $|(G \sm R_i\cup R^*)_{B_i\cup B^*}|$ are the elements of the $i^{th}$ row of $\phi([G])$ and, equivalently, the $|(H \sm  R\cup R_l^*)_{B\cup B_l^*})|$ are the elements of the $l^{th}$ column of $\phi([H])$. 

Then, we see that 
$$m_{il}=\sum_p \phi([G])_{ip}E_{j_{pp}}\phi([H])_{pl}$$
which shows that $\phi$ is a homomorphism of monoids from the $j-$equivalent graphs under the $j-$join operation to the $k\times k$ matrices under the sandwich operation with sandwich element $E_j$. Therefore we can define $$\phi([G])\bullet\phi([H])=\phi([G]) E_j \phi([H]) =\phi([G]\jjoin [H])=\phi([G\jjoin H]). $$
\end{proof}

\begin{theorem}\label{1to1}
The homomorphism $\phi:(\mathcal{G}_j,\jjoin)\to (M(k,\Z),E)$ is one-to-one where $(M(k,\Z),E)$ is the sandwich monoid with sandwich element $E$ as described above. 
\end{theorem}
\begin{proof}   
  This proof follows straight from definitions. If $[G_1], [G_2]\in \mathcal{G}_j$ such that $[G_1]\neq [G_2]$, then there exists $i,j$ such that $\det((r_i\circ c_l) (G_1))\neq \det((r_i\circ c_l) (G_2)).$ But then, this implies that $\phi(G_1)\neq \phi(G_2)$ as matrices. 
\end{proof}

The following lemma gives us a matrix representation for this monoid.

\begin{lemma}\label{Sandiso}[Hickey, \cite{JBH}]
Let $S$ be a semigroup and let $a\in S$. If the semigroup $(S, a)$ has identity element 1 then 
\begin{enumerate}[i.]
    \item $S$ has identity element 
    \item the elements $a, 1$ lie in the unit group of $S$ and are inverse to each other,
    \item $(S,a)\cong S$
\end{enumerate}
\end{lemma}

Let $\gamma:(M(k,\Z),E)\to (M(k,\Z))$ be the isomorphism from part $iii$ from the lemma above. Then, define $\Phi = \gamma\circ\phi$. \Cref{main-thm} above implies that $(\Phi(\mathcal{G}_j), \cdot)$ is a submonoid of $(M(k,\Z),\cdot)$. 
\begin{corollary}
    The monoid $(\mathcal{G}_j,\jjoin)$ is isomorphic to a submonoid of $(M(k,\Z),\cdot)$
\end{corollary}

It remains an open question what the structure of this submonoid is. The next two sections will begin to explore the structure.

\subsection{The $[0]$ and $[n]$ classes}
We generalize the identity class by defining the $[n]_j$, $n\in\{0,1,2,...\}$ to have the properties such that if we join it to any graph on the left or right, then it multiplies the determinant of the graph by $n$. That is, for any $N\in [n]_j$ and any $G, H\in\mathbb{G}_j$, $ \phi([G\jjoin N\jjoin H])=n\phi([G\jjoin H])=n\phi([G])E_j\phi([H])$. Hence, $N$ must have the property that $\phi([N])=nE_j$. Which gives us the following necessary and sufficient conditions on the graph $N$:

\begin{theorem}\label{nclass}
Let $j\in\N$,  $n\in \N\cup\{0\}$ and $N\in \mathbb{G}_j$. Then, $N\in [n]_j$ if and only if $N$ satisfies the condition that for any vertex deletion sets $R_1,R_2$ and corresponding allowable handle sets $B_1, B_2$ and their conjugate sets, $$|(N\sm{R_1\cup R_2^*})_{B_1\cup B_2^*}|=
            \begin{cases} 
              n(-1)^{|R_1|+|B_1|} &\text{if } R_1=R_2, B_1=B_2 \\
              0&\text{ otherwise}.
            \end{cases}$$
\end{theorem}

We observe that this implies that 
\begin{enumerate}
    \item $|N|$=n
    \item For any vertex deletion set $R$ and handle set $B$ and their conjugate sets $R^*, B^*$ where at least one set is nonempty,
    $$|(N\sm R)_B|=|(N\sm R^*)_{B^*}|=0$$
\end{enumerate}

An example of a graph with the properties above is similar to the example of the identity graph given in \Cref{jiden}. We use the fact that the determinant of the complete graph is $ |K_{n+1}| =  (-1)^{n}(n)$. We need to include an extra copy of $P_2$ if $n+j$ is odd so that the resulting determinant is $n$. 
\begin{equation*}
N=
\begin{cases} 
      K_{n+1} \bigsqcup \left(\ds{\coprod_{i=1}^j P_2}\right )  & \text{if }n+j\text{ is even} \\
      K_{n+1} \bigsqcup \left(\ds{\coprod_{i=1}^{j+1} P_2}\right ) & \text{if }n+j\text{ is odd} 
   \end{cases}
\end{equation*}

The fact that these are the only graphs is proven in the same way that the uniqueness of the identity was proven.

\subsection{Group Structure in the Semigroup}
Consider $[G]\in \mathcal{G}j$ such that there exists $[G^{-1}] \in \mathcal{G}_j$ so that $[G\jjoin G^{-1}]=Id_j$. Then $\phi([G\jjoin G^{-1}])=\phi([G])E_j\phi([G^{-1}])=E_j$. If we take the determinant of both sides, this gives us 

$$\det(\phi([G]))\det(\phi([G^{-1}]))=1$$

Which means that all of these graphs map to invertible matrices with determinants $\pm1$. So, if $L$ is the set of all classes of graphs in $\mathcal{G}_j$ that have inverses, that is, the largest group in the semigroup $\mathcal{G}_j$. Then  $\Phi(L)\subset GL \left( {2j\choose j},\Z \right)$.  We know that $GL(n,\Z)$ is finitely generated, so if we can find graphs that map to the generators then we will have a complete representation of this group. 

\section{Applications}
In this section we will show how we can use the homomorphism to easily calculate the $n$-fold join of various graphs. By work in \cite{BP}, the following theorem holds.

\begin{theorem}\label{k-graphs}
Let $m$ and $n$ be positive integers and $K_m$ the complete graph on $m$ vertices. Then

$$
\left|K_m \njoin{j} K_n\right| = \left\{ 
\begin{array}{ll}
(-1)^{m+n}(m+n-3), & \text{for }j=1 \\
0, & \text{for }j \ge 2.
\end{array}
\right.
$$
\end{theorem}

We will denote the the \B{$n$-fold $j$-join} of a graph $G$ by $ \nfoldjoin{n}{j} G$. That is,
$$ 
\nfoldjoin{n}{j} G= \underbrace{G \jjoin G \jjoin \cdots \jjoin G}_{n \text{-joins}}.
$$
For example, the $0$-fold $j$-join of a graph $G$ is just the graph itself, the $1$-fold $j$-join of $G$ is $G \njoin{j} G$, and the $2$-fold $j$-join is $G \njoin{j} G \njoin{j} G$.
\subsection{1-join Example}

\begin{theorem}\label{j-join-complete}
For integers $n,m$ where $n \geq0$ and $m\geq3$, 

$$
\phi\left(\nfoldjoin{n}{1} K_m\right) = (-1)^{mn+m+n}
\begin{bmatrix} -\left[ (m-2)n+(m-1) \right] & (n+1)(m-2)\\
(n+1)(m-2) & -[(m-2)n+(m-3)] 
\end{bmatrix}
$$

\end{theorem}

Note that this will immediately imply the following result by taking
the $(1,1)$ element of the matrix given by $\phi\left( \nfoldjoin{n}{1} K_m\right)$.

\begin{corollary}
For integers $n,m$ where $n \geq0$ and $m\geq3$, the determinant of the $n-$fold join of the the complete graph on $m$ vertices, with any labeling, is given by 
$$
\left| \nfoldjoin{n}{1} K_m \right| = (-1)^{(m+1)n}\left[ (m-2)n+m-1 \right].
$$
\end{corollary}

\begin{proof}[Proof of \Cref{j-join-complete}]
The proof is by induction on $n$. 
For $n=0$ we have, 

\begin{align*}
    \phi\left( \nfoldjoin{0}{1} K_m\right) = \phi\left( K_m \right) &=
    \begin{bmatrix}
    |K_m| & |K_m \sm \{m\}| \\
    |K_m \sm \{1\} & |K_m \sm \{1,m\}
    \end{bmatrix}\\
    &= (-1)^{m-1}
    \begin{bmatrix}
    m-1 & -(m-2) \\
    -(m-2) & m-3
    \end{bmatrix}.
\end{align*}

Assuming the result for $n$, we now examine $n+1$, noting that for the 1-join, $E_1$ from \Cref{main-thm} is the matrix $\begin{bmatrix} 1 & 0 \\ 0 & -1 \end{bmatrix}$. Then 
$$\phi \left( \nfoldjoin{n+1}{1} K_m\right) = \phi \left( \left(\nfoldjoin{n}{1} K_m\right) \njoin{1} K_m \right) = \phi \left( \left(\nfoldjoin{n}{1} K_m\right) \right)\cdot  E_1 \cdot \phi(K_m)$$ 
can be expressed as

\fontsize{9}{9}\selectfont
$
\begin{bmatrix} -\left[ (m-2)n+(m-1) \right] & (n+1)(m-2)\\
(n+1)(m-2) & -[(m-2)n+(m-3)] 
\end{bmatrix}
\cdot
\begin{bmatrix}
1 & 0\\
0 & -1
\end{bmatrix}
\cdot 
\begin{bmatrix}
m-1 & -(m-2) \\
-(m-2) & m-3
\end{bmatrix}
$
\fontsize{11}{11}\selectfont

scaled by $(-1)^{mn+n+1}$. A straightforward calculation gives the result.
\end{proof}

\subsection{2-join Example}

\begin{theorem}
For integers $n,m$ where $n \geq1$ and $m\geq3$, 
$$
\left| \nfoldjoin{n}{2} K_m \right| = 0.
$$
\end{theorem}

\begin{proof}
Note that $\phi(K_m)$ and $E_2$ are, respectively the matrices

$$(-1)^m\begin{bmatrix}
 - (-1+m) &  (-2+m) &  (-2+m) &- (-3+m) & 1 & 1 \\
  (-2+m) & -(-3+m) & - (-3+m) &  (-4+m) & -1 & -1 \\
  (-2+m) & - (-3+m) &-(-3+m) &  (-4+m) & -1 & -1 \\
 -(-3+m) &  (-4+m) &  (-4+m) & -(-5+m) & 1 & 1 \\
 1 & -1 & -1 & 1 & 0 & 0 \\
 1 & -1 & -1 & 1 & 0 & 0 \\
\end{bmatrix},$$ and

$
E_2=\begin{bmatrix}
 1 & 0 & 0 & 0 & 0 & 0 \\
 0 & -1 & 0 & 0 & 0 & 0 \\
 0 & 0 & -1 & 0 & 0 & 0 \\
 0 & 0 & 0 & 1 & 0 & 0 \\
 0 & 0 & 0 & 0 & -1 & 0 \\
 0 & 0 & 0 & 0 & 0 & -1 \\
\end{bmatrix}
$

Then by direct calculation we see that
$$\phi\left( K_m \njoin{2} K_m \right) = \phi(K_m) \cdot E_2 \cdot \phi(K_m)$$ is the zero matrix, which implies the result.
\end{proof}

We note that the above theorem has significance despite the previously known result that $|K_m \njoin{j} K_m| = 0$. It can be the case that $|G|=0$ with $|G \njoin{j} H| \ne 0$, as can be seen by taking $G = P_3, H=P_5,$ with the canonical labeling and $j = 1$.
\subsection{Chain of path graphs with different labeling}

As we have discussed in this paper, the labeling of the graph contains the information on where to join the graph. 
\begin{example} \label{chainofp4s}
Let $\widetilde{P}_4$ be a path graph on 4 vertices where we label them starting at one end as $1,3,2,4$. 
\begin{figure}[H]
 \begin{center}
        \begin{tikzpicture}[scale=1]
            \tikzstyle{every node}=[circle,fill=black,inner sep=0.5mm,yshift=0,xshift=0]
            \node  [label={[xshift=0cm ]left:4}] (4+0) at  (0,4)  {};
            \node [label={[xshift=0cm ]left:2}](2+0) at (0,3) {};
            \node [label={[xshift=0cm ]left:3}](3+0) at (0,2) {};
            \node [label={[xshift=0cm ]left:1}] (1+0) at (0,1) {};
            \draw (1+0)--(4+0);
        \foreach \x in {-3,3}{
            \node  [label={[xshift=0cm ]left:1}] (1+\x) at  (\x,4)  {};
            \node [label={[xshift=0cm ]left:3}](3+\x) at (\x,3) {};
            \node [label={[xshift=0cm ]left:2}](2+\x) at (\x,2) {};
            \node [label={[xshift=0cm ]left:4}] (4+\x) at (\x,1) {};
            \draw (1+\x)--(4+\x);
        }
        \draw (4+-3)--(1+0);
        \draw (3+-3)--(2+0);
        \draw (3+0)--(2+3);
        \draw (4+0)--(1+3);
        \draw[dashed](3,3)--(5,3); 
        \draw[dashed](3,1)--(5,1); 
        \end{tikzpicture}
    \end{center}
\caption{Chain of $\widetilde{P}_4$}
    
\end{figure}
\end{example}

\begin{theorem}\label{2-join-path}
For integer $n \geq0$, 

$$
\left|\nfoldjoin{n}{2} \widetilde{P}_4\right| = 1
$$
\end{theorem}
\begin{proof}
    Using Harary, we calculate
$$\phi(\widetilde{P}_4)=
\begin{bmatrix}
 1 & 0 & 0 & 0 & 0 & 0 \\
 0 & -1 & -1 & 0 & -1 & -1 \\
 0 & -1 & 0 & 0 & 0 & 0 \\
 0 & 0 & 0 & 1 & 0 & 0 \\
 0 & -1 & 0 & 0 & 0 & -1 \\
 0 & -1 & 0 & 0 & -1 & 0 \\
\end{bmatrix}
$$
Then, since the first row and column of both $\phi(\widetilde{P}_4)$ and $E_2$ is a 1 at $(1,1)$ element and zero elsewhere then the $(1,1)$ element of the product will always be 1.
\end{proof}

\section{Future Work}
Much work remains to be done in this area. The structure of the semigroup is unknown as it depends on the surjectivity of $\phi$. That is, will any integer matrix be the image of a graph under $\phi$? A possible first step is to explore the generators and relators of the relevant submonoids.

 As noted in \cite{SB}, in 1957 Collatz and Sinogowitz in \cite{CS} posed the problem of characterizing graphs with positive nullity, and that an adjacency matrix with zero determinant ensures a graph has positive nullity. Our work demonstrates that joining any collection of graphs where one of them is the $[0]$ class ensures that the joined graphs have a determinant of zero. This gives a whole class of graphs that have positive nullity.

Parallel work can be done using a different matrix other than the adjacency such as the Laplacian, Hermitian, etc. The main obstacle here is to recreate a similar sum decomposition for the determinant of join of graphs as done in \cite{BP}. The authors have preliminary work done in this direction for the Laplacian matrix.

From the point of view of calculations, the techniques shown above can be applied to find generating functions for the determinant of any chain of graphs joined in various ways as done in \Cref{chainofp4s}.

\typeout{}

\thanks{Department of Mathematics and Statistics, Georgia Gwinnett College\\dpinzon@ggc.edu, dpragel@ggc.edu, jroberts@ggc.edu}

\end{document}